\documentclass[oneside,11pt,reqno]{amsart}
\pdfoutput=1

\usepackage{amssymb,amscd,amsmath,latexsym}
\usepackage[mathcal]{euscript}
\usepackage{hyperref}
\usepackage{xcolor}
\usepackage{ytableau}
\usepackage{tikz}

\newtheorem{theorem}{Theorem}[section]
\newtheorem{corollary}[theorem]{Corollary}
\newtheorem{lemma}[theorem]{Lemma}
\newtheorem{proposition}[theorem]{Proposition}
\theoremstyle{definition}
\newtheorem{definition}[theorem]{Definition}

\numberwithin{equation}{section}
 \numberwithin{figure}{section}

\newcommand{\sgn}{\operatorname{sgn}}

\theoremstyle{remark}
\newtheorem{remark}[theorem]{Remark}

\begin{document}
\title[Mullineux fixed]{Generating functions for fixed points of the Mullineux map}
\author{\sc David J. Hemmer}
\address{Department of Mathematical Sciences\\ Michigan Technological University \\
1400 Townsend Drive\\Houghton, MI~49931, USA}
 \email{djhemmer@mtu.edu}

\date{\today}


\begin{abstract}
Mullineux defined an involution on the set of $e$-regular partitions of $n$. When $e=p$ is prime, these partitions label irreducible symmetric group modules in characteristic $p$. Mullineux's conjecture, since proven, was that this ``Mullineux map" described the effect on the labels of taking the tensor product with the one-dimensional signature representation. Counting irreducible modules fixed by this tensor product is related to counting irreducible modules for the alternating group $A_n$ in prime characteristic. In 1991, Andrews and Olsson worked out the generating function counting fixed points of Mullineux's map when $e=p$ is an odd prime (providing evidence in support of Mullineux's conjecture). In 1998, Bessenrodt and Olsson counted the fixed points in a $p$-block of weight $w$. We extend both results to arbitrary $e$, and determine the corresponding generating functions. When $e$ is odd but not prime the extension is immediate, while $e$ even requires additional work and the results, which are different, have not appeared in the literature.
\end{abstract}
\maketitle

\section{Introduction}

Let $S_n$ denote the symmetric group on $n$ letters. Recall that a partition $\lambda=(\lambda_1, \lambda_2, \ldots, \lambda_s) \vdash n$ is $e$\emph{-regular} if no part repeats $e$ or more times. Let $k$ be an algebraically closed field of characteristic $p$. The irreducible $kS_n$ modules are labelled by $p$-regular partitions \cite[Chapter 11]{Jamesbook} and are denoted $\{D^\lambda \mid \lambda \textrm{ is p-regular}\}$. Define an involution $P$ on the $p$-regular partitions by $D^\lambda \otimes \sgn \cong D^{P(\lambda)}.$ In \cite{MullineuxBijectionsofpregular}, Mullineux defined a combinatorial map $\lambda \rightarrow m_e(\lambda)$ on the set of $e$-regular partitions. He conjectured for $e=p$ a prime that $m_p=P$. Almost twenty years later, Kleshchev \cite{KleshchevBranchingIII} described $P$ and Ford and Kleshchev \cite{FordKleshMullineux} proved the conjecture by proving Kleshchev's description matched Mullineux's. The bijection $m_e$ is defined combinatorially on $e$-regular partitions for $e$ arbitrary, and can be interpreted similarly using irreducible representations of a certain Hecke algebra \cite[Chapter 6.3]{Mathasbook}. For $e=2$ the sign representation is trivial and the Mullineux map is the identity. Combinatorial properties of this map have inspired much research, often independent of representation theoretic applications. Recently there has been a flurry of research in this area, including multiple new descriptions of the Mullineux map \cite{FayersCrystalsRegularisation, JaconMullineuxandcrystal}, so having complete information about fixed points in each $e$-block for arbitrary $e$ is particularly timely.

Let $A_n$ be the alternating group and let the characteristic of $k$ be an odd prime $p$. One can count irreducible $kA_n$ modules in two ways: group theoretically by counting $p$-regular conjugacy classes of $A_n$, or by restricting modules from $S_n$ using Clifford theory and counting in terms of fixed points of the  map $P$. In 1991 Andrews and Olsson, using Olsson's work in \cite{OlssonTheNumberofModularCharacters}, counted the fixed points of $m_p$:

\begin{theorem}\label{thm: Olsonfixedpoint}\cite[Propositions 2,3]{AndrewsOlssonPartitionIdentities}
Let $p>2$ be prime. The number of fixed points of $m_p$ is the number of partitions of $n$ with distinct odd parts, none of which are divisible by $p$.
\end{theorem}

This answer agreed with the known representation theoretic count for fixed points of $P$, providing evidence for the as-yet-unproven Mullineux conjecture by showing $m_p$ had the expected number of fixed points. Later Bessenrodt and Olsson refined this by computing fixed points in an arbitrary $p$-block of weight $w$:

\begin{theorem}\cite[Theorem 3.5]{BessendrodtOlssonOnResidueSymbols}
\label{thm: BessOlssonFixedpointsbyweight}Let $p$ be an odd prime, and $\mu \vdash n-pw$ be a self-conjugate $p$-core. For $w$ even, the number of $\lambda \vdash n$ with $m_p(\lambda)=\lambda$ and having $p$-core $\mu$  is given by the number of $(p-1)/2$-tuples of partitions with total weight $w/2$. For $w$ odd, there are no such partitions.

\end{theorem}

Our main results are to extend Theorems \ref{thm: Olsonfixedpoint} and \ref{thm: BessOlssonFixedpointsbyweight} to arbitrary $e$. For $e$ odd but not prime it is a simple observation that the original proofs carry over, while for $e$ even additional work is needed, including dealing with the unusual setting of using $t$-bar cores for even $t$. The author would like to acknowledge his colleague William Keith for useful discussions about generating functions. He would also like to thank the two anonymous referees for a careful reading of the paper and useful suggestions.

\section{The Mullineux Map, $e$-cores and $e$-weights}

 For a partition $\lambda=(\lambda_1, \lambda_2, \ldots, \lambda_s)$ with Young diagram $[\lambda]$, define the \emph{rim} of $\lambda$ to be the boxes along the southeast edge of the diagram, i.e. boxes $(i,j)\in [\lambda]$ with $(i+1,j+1)\not\in [\lambda]$. Now consider a subset of the rim defined as follows, and called the $e$-rim. Starting at the top right of the rim, take the first $e$ elements on the rim. Then move to the rightmost element of the rim in the next row, and take the next $e$ elements. Continue until the final row is reached, observing that the final segment may contain fewer than $e$ boxes.

For example if $\lambda=(7,7,7,4,4,1,1)\vdash 31$ and $e=5$ we have the $e$-rim:\\

\begin{figure}[h]
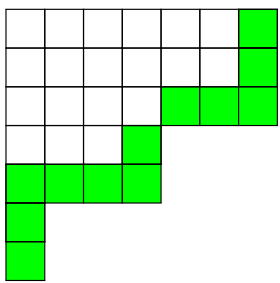

  \centering

\ytableausetup{centertableaux, boxsize=.5cm}
\ytableaushort
{}
*{7,7,4,4,1,1}
*[*(green)]{6+1,6+1,4+3,3+1,4,1,1}\\
 \caption{The $5$-rim of $\lambda=(7,7,7,4,4,1,1)$.}
\label{fig:defineerim}
\end{figure}
Let $a_1$ be the number of boxes and $r_1$ be the number of rows  in the $e$-rim, so in our example $a_1=12$, $r_1=7$.

To define the Mullineux symbol $G_e(\lambda)$, remove the $e$-rim, and then calculate the $e$-rim of what remains to determine $(a_2, r_2)$. Continue this process until all boxes are removed. Assemble these numbers in an array, called the \emph{Mullineux symbol} of $\lambda$:
\begin{equation}
\label{eq: def mullineux symbol}
G_e(\lambda)=\left(
                 \begin{array}{cccc}
                   a_1 & a_2 & \cdots & a_k \\
                   r_1 & r_2 & \cdots & r_k \\
                 \end{array}
               \right).
               \end{equation}

Notice that $(a_1,a_2,\ldots, a_k)$ is also a partition of the same integer $n$. For example from Figure \ref{fig:examplemullsymbol} we see that:

$$G_5(7,7,7,4,4,1,1)=\left(
                 \begin{array}{ccccc}
                   12 & 8 & 5&4&2 \\
                        7&4&3&3&2\\
                 \end{array}
               \right).$$

\begin{figure}
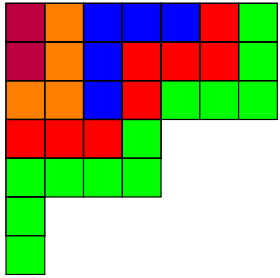

  \centering
  \ytableausetup{centertableaux, boxsize=.5cm}
\ytableaushort
{}
*{7,7,4,4,1,1}
*[*(green)]{6+1,6+1,4+3,3+1,4,1,1}
*[*(red)]{5+1,3+3,3+1,0+3}
*[*(blue)]{2+3,2+1,2+1,0+3}
*[*(orange)]{1+1,1+1,0+2}
*[*(purple)]{0+1,0+1}
  \caption{Calculating the Mullineux symbox $G_5(7,7,7,4,4,1,1)$}
  \label{fig:examplemullsymbol}
\end{figure}

Now define $\epsilon_i=0$ if $e \mid a_i$ and $\epsilon_i=1$ otherwise, and set $s_i=a_i-r_i+\epsilon_i.$
Then we have Mullineux's conjecture (now theorem):

\begin{proposition}
\label{prop:MullMap}
  Let $\lambda \vdash n$ be $e$-regular with Mullineux symbol

\begin{equation*}
G_e(\lambda)=\left(
                 \begin{array}{cccc}
                   a_1 & a_2 & \cdots & a_k \\
                   r_1 & r_2 & \cdots & r_k \\
                 \end{array}
               \right).
\end{equation*} Then the Mullineux symbol for $P(\lambda)$ is:

\begin{equation*}
\label{eq: mullformlambda}
G_e(P(\lambda))=\left(
                 \begin{array}{cccc}
                   a_1 & a_2 & \cdots & a_k \\
                   s_1 & s_2 & \cdots & s_k \\
                 \end{array}
               \right).
\end{equation*}
\end{proposition}

It is easy to reconstruct $\lambda$ from $G_e(\lambda)$ so Proposition \ref{prop:MullMap} gives a combinatorial description of the Mullineux map. For example with $\lambda$ as in Figure \ref{fig:examplemullsymbol}
$$G_5(m_5(7,7,7,4,4,1,1))=\left(
                 \begin{array}{ccccc}
                   12 & 8 & 5&4&2 \\
                        6&5&2&2&1\\
                 \end{array}
               \right)$$
               and $m_5(7,7,7,4,4,1,1)=(12,9,4,2,2,2).$

In his original paper Mullineux described necessary and sufficient conditions for such an array to arise as the Mullineux symbol of an $e$-regular partition:

\begin{proposition}\cite[Theorem 3.6]{MullineuxBijectionsofpregular}\label{prop:Mull conditions} An array as in \eqref{eq: def mullineux symbol} is the Mullineux symbol of an $e$-regular partition of $n=\sum a_i$ if and only if:
\begin{itemize}
  \item[(i)] $0 \leq r_i-r_{i+1} \leq e$;
  \item[(ii)] $r_i-r_{i+1} + \epsilon_{i+1} \leq a_i-a_{i+1} \leq r_i-r_{i+1} + \epsilon_{i+1}+e$;
  \item[(iii)]$r_i=r_{i+1} \Rightarrow e\mid a_i$;
  \item[(iv)]$r_i-r_{i+1}=e \Rightarrow e \nmid a_i$;
   \item[(v)]  $0 \leq a_k-r_k<e$
  \item[(vi)] 1 $\leq r_k \leq e$ and if $r_k=e$ then $a_k-r_k>0.$

\end{itemize}
\end{proposition}
Thus, fixed points of $m_e$ correspond to Mullineux symbols of the form:
\begin{equation}
\label{eq: MullSymbolforfixedpoint}
G_e(\lambda)=\left(
                 \begin{array}{cccc}
                   a_1 & a_2 & \cdots & a_k \\
                   \frac{a_1+\epsilon_1}{2} &\frac{a_2+\epsilon_2}{2} & \cdots & \frac{a_k+\epsilon_k}{2} \\
                 \end{array}
               \right)
\end{equation} where $a_i$ is even if and only if $e \mid a_i$ if and only if $\epsilon_i=0$.
Now we can apply Proposition \ref{prop:Mull conditions} to arrays of the form \eqref{eq: MullSymbolforfixedpoint}, i.e. fixed points. The conditions on the $r_i$ in Proposition \ref{prop:Mull conditions} can easily be translated to conditions on the $a_i$. Thus we can enumerate fixed points simply by counting the suitable partitions $(a_1, a_2, \ldots, a_k) \vdash n$.

\begin{definition}
\label{def: Me(n)=fixedpointsmullineux}
Define $\mathcal{M}_e(n)$ to be the set of partitions $(a_1, a_2, \ldots, a_k) \vdash n$ satisfying:
\begin{itemize}
                                                                   \item[(i)] $2 \mid a_i \Longleftrightarrow e \mid a_i$
                                                                   \item[(ii)] $0 \leq a_i-a_{i+1} \leq 2e$
                                                                   \item[(iii)] If $a_i=a_{i+1}$ then $a_i$ is even.
                                                                   \item[(iv)] If $a_i-a_{i+1} =2e$ then $a_i$ is odd.
                                                                   \item[(v)]$a_k<2e$.
                                                                 \end{itemize}

\end{definition}
Then we have:
\begin{proposition}\cite[Proposition 1] {AndrewsOlssonPartitionIdentities}
\label{prop: AO conditions on mullineux fixed}
Let $p$ be an odd prime. The number of partitions $\lambda \vdash n$ fixed by the Mullineux map is equal to the cardinality of $\mathcal{M}_p(n)$.

\end{proposition}

  The proof of Proposition \ref{prop: AO conditions on mullineux fixed} goes through for $e$ arbitrary, so to enumerate Mullineux fixed points we need to understand the set $\mathcal{M}_e(n)$.
The criteria defining $\mathcal{M}_e(n)$ are not easily translated into a generating function.  The main theorem (Theorem 2) of \cite{AndrewsOlssonPartitionIdentities} gives a remarkable enumeration of partitions with difference conditions like this. Observe for $p$ odd that $\mathcal{M}_p(n)$ is a set of type $P_2(A;N,n)$ from that paper, where $N=2p$ and $A=\{1,3,5, \ldots p-2,p+2,\ldots 2p-1\}$.  The paper gives a bijection with a set $P_1(A; 2p, n)$, which in this case is just partitions with distinct odd parts not divisible by $p$, giving Theorem \ref{thm: Olsonfixedpoint}.

\subsection{$e$-cores and $e$-weights}
\label{subsec:ecoreseweights}
A \emph{rim e-hook} of $\lambda$ is $e$ consecutive boxes in the rim which, when removed, leave a Young diagram of a partition.  Given a partition $\lambda \vdash n$, there is a unique partition $\lambda_{(e)}$, called the $e$-core of $\lambda$, and obtained by removing rim $e$-hooks from $\lambda$ until none remain. The number of such hooks removed is called the $e$-weight $w$ of $\lambda$, so $\lambda_{(e)} \vdash n-ew$. This is all described in \cite{JamesKerberbook}. In \cite[Bijection 2]{GarvanKimStantonCranksandtcores} there is a bijection between $e$-core partitions and vectors $\vec{n}=(n_0,n_1, \ldots, n_{e-1}) \in \mathbb{Z}^e$ where $\sum_{i=0}^{e-1}n_i=0$. Given the vector $\vec{n}$, the corresponding $e$-core is a partition of
\begin{equation}\label{eq: weightofpcorefornvector}
  \frac{e}{2}\sum_{i=0}^{e-1}n_i^2 + \sum_{i=0}^{e-1}in_i.
\end{equation}
The $e$-weight can be obtained  by subtracting \eqref{eq: weightofpcorefornvector} from $\sum_{j}\lambda_j$ and dividing by $e$.

Bessenrodt and Olsson worked out how to determine $\vec{n}$ directly from the Mullineux symbol of a Mullineux fixed point (they state it for $p$ an odd prime but the proof works for any $e$).

\begin{proposition}\cite[p. 235]{BessendrodtOlssonOnResidueSymbols}
  \label{prop: BOnvectorfromMull}
  Suppose $\lambda \vdash n$ with $m_e(\lambda)=\lambda$, $e$-core $\lambda_{(e)}$ and $G_e(\lambda)$ is as in \eqref{eq: MullSymbolforfixedpoint}. Then the vector $\vec{n}$ corresponding to $\lambda_{(e)}$ is:
  \begin{equation}\label{eq:njfromMullSym}
   n_j= \#\{i \mid \frac{a_i-\epsilon_i}{2} \equiv j \mod e\} - \#\{i \mid \frac{-a_i-\epsilon_i}{2} \equiv j \mod e\}.
  \end{equation}
\end{proposition}

\begin{remark}
  \label{rmk:aidivbyedonocontribute}
  It is clear from \eqref{eq:njfromMullSym} that $\vec{n}$, and thus $\lambda_{(e)}$, depends only on the set of nonzero residues mod $e$ of the $a_i$. The $a_i$ which are divisible by $e$ (and hence even with $\epsilon_i=0$), do not contribute to $\vec{n}.$
\end{remark}

\section{Main Results}

 In Table \ref{tab:e=4fixedpointsbyweight} we list the number of fixed points under the Mullineux map when $e=4$ for $0 \leq n \leq 20$ and $e$-weights $0 \leq w \leq 5$. Weight zero partitions are $e$-cores and $m_e$ acts as conjugation on $e$-cores, so the first column of this table enumerates self-conjugate 4-core partitions, which is sequence A053692 in the Online Encyclopedia of Integer Sequence (OEIS) \cite{oeis}. Each successive column is shifted down by $e=4$ and multiplied by the corresponding entry in the sequence $\{1,1,3,4,9,12,23\ldots \}$, which is A002513 in the OEIS, counting partitions of $n$ with even parts of two colors, also known as ``cubic partitions". Our main results generalize Theorems \ref{thm: Olsonfixedpoint} and \ref{thm: BessOlssonFixedpointsbyweight} to arbitrary $e$ and explain the structure of this table.

\begin{table}[h]
\centering
\caption{Mullineux fixed points for $e=4$ by weight}

\begin{tabular}{|c|c|c|c|c|c|c|}
\hline
$n$ & $w=0$ & $w=1$ & $w=2$ & $w=3$ & $w=4$ & $w=5$ \\
\hline

0 & 1 &  &  &  &  &  \\
1 & 1 &  &  &  &  &  \\
2 & 0 &  &  &  &  &  \\
3 & 1 &  &  &  &  &  \\
4 & 1 & 1 &  &  &  &  \\
5 & 1 & 1 &  &  &  &  \\
6 & 1 & 0 &  &  &  &  \\
7 & 1 & 1 &  &  &  &  \\
8 & 0 & 1 & 3 &  &  &  \\
9 & 0 & 1 & 3 &  &  &  \\
10 & 2 & 1 & 0 &  &  &  \\
11 & 0 & 1 & 3 &  &  &  \\
12 & 1 & 0 & 3 & 4 &  &  \\
13 & 1 & 0 & 3 & 4 & & \\
14 & 1 & 2 & 3 & 0 &  &  \\
15 & 2 & 0 & 3 & 4 &  &  \\
16 & 0 & 1 & 0 & 4 & 9 &  \\
17 & 0 & 1 & 0 & 4 & 9 &  \\
18 & 1 & 1 & 6 & 4 & 0 &  \\
19 & 1 & 2 & 0 & 4 & 9 &  \\
20 & 0 & 0 & 3 & 0 &  9 &12\\
\hline
\end{tabular}
\label{tab:e=4fixedpointsbyweight}
\end{table}

We fix some standard generating function notation. Define the $q$-Pochhammer symbol:
$$(a;q)_\infty := \prod_{k=0}^\infty(1-aq^k)$$ and the Ramanujan $\chi$ function:
$$\chi(q):=(-q,q^2)_\infty=(1+q)(1+q^3)(1+q^5) \cdots.$$
Observe that $\chi(q)$ is the generating function counting partitions of $n$ into distinct odd parts.\\

Let $mf_e(n)$ be the number of $e$-regular partitions of $n$ fixed by the Mullineux map and let $$MF_e(q):=\sum_{n=0}^{\infty}mf_e(n)q^n$$ be the corresponding generating function. Our first result determines this generating function:

\begin{theorem}
\label{thm: main result}
\begin{itemize}
  \item[(a)]\cite[for $e$ prime]{OlssonTheNumberofModularCharacters} Let $e$ be odd. Then $mf_e(n)$ is the number of partitions of $n$ into distinct odd parts not divisible by $e$. Thus:

  \begin{equation}\label{eq: e odd generating function}
    MF_e(q)=\frac{\chi(q)}{\chi(q^e)}= \frac{(1+q)(1+q^3)(1+q^5) \cdots}{(1+q^e)(1+q^{3e})(1+q^{5e})\cdots}=\prod_{\substack{k \text{ odd}\\ e \nmid k}}(1+q^k).
  \end{equation}
  \item[(b)] Let $e$ be even. Then $mf_e(n)$ is the number of partitions of $n$ into  parts which are either odd or are odd multiples of $e$, and such that the odd parts are all distinct. Thus:

  \begin{equation}\label{eq: e even generating function}
    MF_e(q)=\frac{\chi(q)}{\chi(-q^e)}= \frac{(1+q)(1+q^3)(1+q^5) \cdots}{(1-q^e)(1-q^{3e})(1-q^{5e})\cdots}.
  \end{equation}
\end{itemize}

\end{theorem}
Notice  when $e=2$ that $MF_2(q)$ simplifies to $\prod_{k \text{ odd}}\frac{1}{1-q^k}$, the generating function counting partitions with odd parts, which is known to be the same as that for distinct parts. Here, the Mullineux map is trivial and we obtain $mf_2(n)$ is the number of two-regular partitions, i.e. partitions with distinct parts, as expected.

If we look at the corresponding alternating series, there is a nice common description of the two generating functions:

\begin{corollary}
For $e$ arbitrary:
\begin{equation}MF_e(-q)=\sum_{n=0}^{\infty}(-1)^nmf_e(n)q^n = \prod_{k=1}^{\infty}\frac{1+q^{ek}}{1+q^k}.\end{equation}
\end{corollary}
\begin{proof}
For both the $e$ even and $e$ odd case, $MF_e(-q)=\frac{\chi(-q)}{\chi(-q^e)}$, which simplifies to the expression above.
\end{proof}

For $e=3, 4, 5, 6$ the generating function $MF_e(-q)$ corresponds to the sequences A098884, A261734, A133563 and A261736 respectively in the OEIS \cite{oeis}. Before this paper appeared in preprint form, there was no mention of the Mullineux map in any of these entries!

\subsection{Mullineux fixed points of a given $e$-weight}

We also give a result for all $e$ generalizing Theorem \ref{thm: BessOlssonFixedpointsbyweight}, counting Mullineux fixed points by $e$-weight.

\begin{definition}
  Let $mf_{e,w}(n)$ be the number of $\lambda \vdash n$ with $e$-weight $w$ and $m_e(\lambda)=\lambda$.
\end{definition}
It is well-known that the Mullineux map of an $e$-core partition is the transpose or conjugate of that partition. So determining $mf_{e,0}(n)$ means counting self-conjugate $e$-cores. These were enumerated by Garvan, Kim and Stanton.

\begin{definition}
  Let $sc_e(n)$ ($=mf_{e,0}(n)$) denote the number of self-conjugate $e$-core partitions of $n$ and let $$SC_e(q):=\sum_{n=0}^{\infty} sc_e(n)q^n.$$
\end{definition}

Garvan, Kim and Stanton determined $SC_e(q)$:

\begin{theorem}\cite[7.1a,b]{GarvanKimStantonCranksandtcores}
  \label{thm: GKSselfconjugateecores}
  The generating function $SC_e(q)$ is:

    \[
    SC_e(q)) =
    \begin{cases}
      (-q,q^2)_\infty(q^{2e},q^{2e})_\infty ^{e/2} & \mbox{if } e \mbox{ is even} \\[10pt]
      \frac{(-q,q^2)_\infty(q^{2e},q^{2e})_\infty ^{\frac{e-1}{2}}}{(-q^e,q^{2e})}& \mbox{if } e \mbox{ is odd.}
    \end{cases}
  \]
\end{theorem}

It turns out that for a self-conjugate $e$-core partition of $n-ew$ , the number of Mullineux fixed point partitions of $n$ with that $e$ core depends only on the weight $w$. These numbers, for $w=0$, $w=1$, etc. also have nice generating functions, again depending on the parity of $e$. The next definition gives the sequences of ``column multipliers" that we observed in Table \ref{tab:e=4fixedpointsbyweight}:

\begin{definition}
\label{def:multipliers}
  Let $e$ be even. Define $f_e(n)$ by:\begin{eqnarray}
      \label{eq:f_edef}
        \sum_{n=0}^{\infty}f_e(n)q^n &:=& \prod_{k=1}^{\infty}\frac{1}{(1-q^{2k})^{e/2}}\frac{1}{(1-q^{2k-1})} \\
         &=& \nonumber \frac{1}{(q^2,q^2)_\infty ^{e/2} (q,q^2)_\infty}.
      \end{eqnarray}
  Let $e$ be odd. Define $g_e(n)$ by:
      \begin{eqnarray}
        \sum_{n=0}^{\infty}g_e(n)q^n &:=& \prod_{k=1}^{\infty}\frac{1}{(1-q^{k})^{(e-1)/2}} \\
         &=& \nonumber \frac{1}{(q,q)_\infty ^{(e-1)/2} }
      \end{eqnarray}
\end{definition}

Notice that both $f_e(n)$ and $g_e(n)$ enumerate certain tuples of partitions with total weights adding up to $n$. With this notation, we can state our theorem enumerating Mullineux fixed points of a given weight. This result for $e$ an odd prime is Theorem 3.5 of \cite{BessendrodtOlssonOnResidueSymbols}.

\begin{theorem}
  \label{thm: Countbyweight}
  \begin{itemize}
    \item[(a)] Suppose $e$ is even. Then:
    $$mf_{e,w}(n)=f_e(w)sc_e(n-ew).$$
    \item[(b)] Suppose $e$ is odd. Then $mf_{e,w}(n)$ is zero unless $w$ is even in which case
    $$mf_{e,w}(n)=g(\frac{w}{2})sc_e(n-ew).$$
  \end{itemize}
\end{theorem}
 Theorem \ref{thm: Countbyweight}  lets us determine a single two-variable generating function that keeps track of Mullineux fixed points by $n$ and the weight $w$. We will need a small reindexing given by:

 \begin{lemma}
   \label{lem:reindexing}
   Suppose $A(x)=\sum_{i=0}^{\infty}a_ix^i$ and $B(x)=\sum_{j=0}^{\infty}b_jq^j.$ Then:

 $$A(q^ex)B(q)=\sum_{n,w}a_wb_{n-ew}x^wq^n.$$
 $$A(q^{2e}x^2)B(q)=\sum_{n}\sum_{w \textrm{even}}a_{\frac{w}{2}}b_{n-ew}x^wq^n.$$

\end{lemma}
\begin{proof}
  Simply reindex the sum with a substitution $n=j+ew$ and (for the second term), $t=w/2$.
\end{proof}

Now define a two-variable generating function:
\begin{definition}
  \label{def:2variableGF}
  Let $$MF_e(x,q):=\sum_{n=0}^{\infty}\sum_{w=0}^{\infty}mf_{e,w}(n)x^wq^n.$$
\end{definition}

Lemma \ref{lem:reindexing} applied to the expressions in Theorem \ref{thm: Countbyweight}, and using the generating functions from Definition \ref{def:multipliers} and Theorem \ref{thm: GKSselfconjugateecores}, proves the following:

\begin{theorem}
\label{thm:Jointgenerating function}
 The generating function $MF_e(x,q)$ is:

    \[
    MF_e(x,q) =
    \begin{cases}
     \frac{(-q,q^2)_\infty(q^{2e},q^{2e})_\infty^{e/2} }{(q^{2e}x^2,q^{2e}x^2)_\infty ^{e/2} (q^ex,q^{2e}x^2)_\infty }& \mbox{if } e \mbox{ is even} \\[10pt]

  \frac{(-q,q^2)_\infty(q^{2e},q^{2e})_\infty ^{e-1/2}}{(q^{2e}x^2,q^{2e}x^2)_\infty ^{(e-1)/2}(-q^e,q^{2e})}& \mbox{if } e \mbox{ is odd}

    \end{cases}
  \]

\end{theorem}
\section{Proof of the main results}
Theorems \ref{thm: Olsonfixedpoint} and \ref{thm: BessOlssonFixedpointsbyweight} are stated for $p$ an odd prime, because that is where the original representation theory motivation comes from. However there is nothing in either proof that makes use of primality, so the $e$ odd case of Theorems \ref{thm: main result} and \ref{thm: Countbyweight} should be attributed  to those authors.  Thus we will consider only the case where $e$ is even.

When $e$ is odd notice that Definition \ref{def: Me(n)=fixedpointsmullineux}(i) gives that if $e \mid a_i$ then actually $N=2e \mid a_i$. This is not the case for $e$-even, and this means that $\mathcal{M}_e(n)$ in this case is not dealt with by the bijections in \cite{AndrewsOlssonPartitionIdentities}. However in \cite{BessenrodtGeneralisationsofAndrewsOlsson}, Bessenrodt gave a vast generalization that includes this case. So in her notation let $N=2e$ for $e$ even. Choose the sets $A'=\{e\}$ and $A''=\{1,3,5,\ldots, 2e-1\}$ with $A=A'\cup A''.$ One can check that the conditions for $\mathcal{M}_e(n)$ are precisely those defining Bessenrodt's set $P_2(A',A''; N,n).$ Bessenrodt gives a bijection with a set $P_1(A',A''; N,n)$, which are all partitions with parts congruent mod $2e$ to elements of $A$ and repeating parts must be congruent to elements of $A'$. This is precisely the set of partitions with all parts odd or odd multiples of $e$ with the odd parts distinct. This proves Theorem \ref{thm: main result}(b).

Finally, to prove Theorem \ref{thm: Countbyweight}(a), we need to count Mullineux fixed points in a given block. We will follow the idea of \cite[Theorem 3.5]{BessendrodtOlssonOnResidueSymbols}, although the proof is slightly more complicated since $e$ is even.

\begin{definition}
\label{def: kappars}
  Define $\kappa(r,s)$ to be the number of tuples $(\gamma^1, \gamma^2, \ldots, \gamma^r; \tau)$ where the $\gamma^i$ and $\tau$ are partitions, $\tau$ has odd parts, and  $|\tau|+2\sum_{i}|\gamma^i|=s.$
\end{definition}

\begin{theorem}
  \label{thm: kappa}
  Let $e$ be even, let $\mu$ be a self-conjugate $e$-core with $w=\frac{n-|\mu|}{e}$. Then:

  $$\kappa\left(\frac{e}{2},w\right)=\#\{\lambda \vdash n \mid \lambda=m_e(\lambda), \lambda_{(e)}=\mu\}.$$
\end{theorem}

Consider the generating function in \eqref{eq:f_edef}. The first term is counting $e/2$ tuples of partitions with even parts and the second term counts partitions with odd parts. We can equally well count $e/2$ tuples of arbitrary partitions but then double all the parts. So Theorem \ref{thm: kappa} immediately implies Theorem \ref{thm: Countbyweight}(a), and all that remains is to prove Theorem \ref{thm: kappa}.

For the next result we will need the following observation, which is easy to see in \cite{BessenrodtGeneralisationsofAndrewsOlsson} (recalling that $N=2e$.)

\begin{lemma}\label{lem:Bessbijection} Suppose $\lambda \vdash n$ is a Mullineux fixed point with $G_e(\lambda)$ as in \eqref{eq: MullSymbolforfixedpoint} and corresponding partition $(a_1,a_2, \ldots, a_k) \in P_2(A',A''; N,n).$ Suppose the image partition under Bessenrodt's bijection is $(b_1, b_2, \ldots b_s) \in P_1(A',A''; N,n)$. Then both partitions have the same set of nonzero residues mod $2e$ (and thus also mod $e$.)
\end{lemma}
Thus we can use the $b_i$'s to calculate $\vec{n}$ and the corresponding $e$-core and $e$-weight as in Section \ref{subsec:ecoreseweights}. Recall that Bessenrodt's $P_1(A',A''; N,n)$ in this case are partitions of $n$ with all parts either odd or odd multiples of $e$, and the odd parts must be distinct. Henceforth we will consider these as \emph{pairs} of partitions $\{(c_1,c_2, \ldots, c_k), e\mu\}$ where $(c_1,c_2, \ldots, c_k)$ has distinct odd parts and $\mu$ has  odd parts. Notice that the set of nonzero residues mod $e$ in $(c_1,c_2, \ldots, c_k)$ is the same as in $(b_1, b_2, \ldots b_t) $, since the odd multiples of $e$ all have $e$-residue zero. To complete the proof we will need to use the theory of $t$-bar cores and $t$-bar quotients.

\subsection{t-bar cores}
For partitions with distinct parts (sometimes called ``bar partitions"), Morris introduced \cite{MorrisTheSpinRepresentation} the notion of a $t$-bar core. These played a role in Humphrey's determination \cite{HumphreysBlocksofProjective} of blocks of projective representations of symmetric groups in characteristic $p$ similar to the role played by $p$-cores in the theory for $S_n$. We will follow the combinatorics introduced in \cite{MorrisYaseenCombinatorialShiftedYoung}, where a version of James' abacus notation is given for computing $t$-bar cores. Note this is different from the more well-known James' abacus display for arbitrary partitions using sets of $\beta$-numbers.

Given a partition $\lambda$ with distinct parts, we form an abacus with runners lying north to south labelled $\{0,1,2, \ldots, t-1\}$. The bead positions are labelled as below:

$$\begin{array}{ccccc}
    0 & 1 & \cdots & t-2 & t-1 \\
    t & t+1 & \cdots & 2t-2 & 2t-1 \\
    \vdots & \vdots & & \vdots & \vdots
  \end{array}$$
To display $\lambda$ on the abacus place a bead on the abacus corresponding to each part.  For example if $\lambda=(23, 21, 17, 13, 11, 9, 7)$ and $t=6$ we obtain the display in Figure \ref{fig:abacusdisplay}, where we have also included the labels on the runners.

  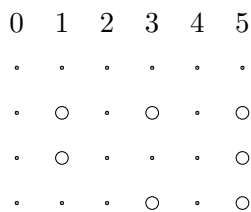
\begin{figure}[h]  \label{fig:abacusdisplay}
\centering

   \begin{tikzpicture}[scale=0.6]
\node at (0,0){0};
\node at (1,0){1};
\node at (2,0){2};
\node at (3,0){3};
\node at (4,0){4};
\node at (5,0){5};
\draw(1,-3) circle(0.14);
\draw(1,-2) circle(0.14);
\draw(3,-2) circle(0.14);
\draw(3,-4) circle(0.14);
\draw(5,-2) circle(0.14);
\draw(5,-3) circle(0.14);
\draw(5,-4) circle(0.14);
\filldraw[color=black, fill=black, very thick](0,-1) circle (0.01);
\filldraw[color=black, fill=black, very thick](0,-2) circle (0.01);
\filldraw[color=black, fill=black, very thick](0,-3) circle (0.01);
\filldraw[color=black, fill=black, very thick](0,-4) circle (0.01);
\filldraw[color=black, fill=black, very thick](1,-1) circle (0.01);
\filldraw[color=black, fill=black, very thick](1,-4) circle (0.01);
\filldraw[color=black, fill=black, very thick](2,-1) circle (0.01);
\filldraw[color=black, fill=black, very thick](2,-2) circle (0.01);
\filldraw[color=black, fill=black, very thick](2,-3) circle (0.01);
\filldraw[color=black, fill=black, very thick](3,-1) circle (0.01);
\filldraw[color=black, fill=black, very thick](3,-3) circle (0.01);
\filldraw[color=black, fill=black, very thick](2,-4) circle (0.01);
\filldraw[color=black, fill=black, very thick](4,-1) circle (0.01);
\filldraw[color=black, fill=black, very thick](4,-2) circle (0.01);
\filldraw[color=black, fill=black, very thick](4,-3) circle (0.01);
\filldraw[color=black, fill=black, very thick](4,-4) circle (0.01);
\filldraw[color=black, fill=black, very thick](5,-1) circle (0.01);

    \end{tikzpicture}

\caption{Abacus display for $\lambda=(23, 21, 17, 13, 11, 9, 7)$ and $t=6$.}

\end{figure}

From $\lambda$ we can obtain its $t$-bar core, denoted $\tilde{\lambda}_{(t)}$, by performing two types of operations on the abacus. The first slides a bead on a runner up one into a vacant spot (beads reaching the location zero disappear). This corresponds to replacing a part $\lambda_i$ in the partition by $\lambda_i-t$ if $\lambda_i-t$ is not already a part, or removing a part $\lambda_i=t$. Alternately you can remove two beads at positions $a$ and $t-a$ in the first row. This corresponds to removing two parts of the partition equal to $a$ and $t-a$. Either operation has the effect of removing a $t$-bar (which we will not define) from the partition and reducing the number being partitioned by $t$. Eventually there are no further operations possible, and you reach the $t$-bar core $\tilde{\lambda}_{(t)}$. The total number of such operations is well-defined and called the $t$-bar weight. So in our example we obtain $\tilde{\lambda}_{(6)}=(9,5,3)$ with abacus display:

  \begin{figure}[h]
\centering

   \begin{tikzpicture}[scale=0.6]
\node at (0,0){0};
\node at (1,0){1};
\node at (2,0){2};
\node at (3,0){3};
\node at (4,0){4};
\node at (5,0){5};

\draw(3,-1) circle(0.14);
\draw(3,-2) circle(0.14);
\draw(5,-1) circle(0.14);

\filldraw[color=black, fill=black, very thick](0,-1) circle (0.01);
\filldraw[color=black, fill=black, very thick](0,-2) circle (0.01);
\filldraw[color=black, fill=black, very thick](0,-3) circle (0.01);
\filldraw[color=black, fill=black, very thick](0,-4) circle (0.01);
\filldraw[color=black, fill=black, very thick](1,-1) circle (0.01);
\filldraw[color=black, fill=black, very thick](1,-2) circle (0.01);
\filldraw[color=black, fill=black, very thick](1,-3) circle (0.01);
\filldraw[color=black, fill=black, very thick](1,-4) circle (0.01);
\filldraw[color=black, fill=black, very thick](2,-1) circle (0.01);
\filldraw[color=black, fill=black, very thick](2,-2) circle (0.01);
\filldraw[color=black, fill=black, very thick](2,-3) circle (0.01);
\filldraw[color=black, fill=black, very thick](3,-3) circle (0.01);
\filldraw[color=black, fill=black, very thick](3,-4) circle (0.01);
\filldraw[color=black, fill=black, very thick](4,-1) circle (0.01);
\filldraw[color=black, fill=black, very thick](4,-2) circle (0.01);
\filldraw[color=black, fill=black, very thick](4,-3) circle (0.01);
\filldraw[color=black, fill=black, very thick](4,-4) circle (0.01);
\filldraw[color=black, fill=black, very thick](5,-2) circle (0.01);
\filldraw[color=black, fill=black, very thick](5,-3) circle (0.01);
\filldraw[color=black, fill=black, very thick](5,-4) circle (0.01);

    \end{tikzpicture}

\caption{Abacus display for $\tilde{\lambda}_{(6)}=(9,5,3)$}
  \label{fig:abacusdisplaycore}
\end{figure}
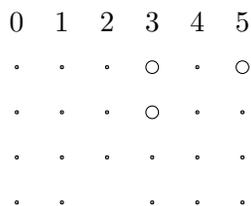

Much of the work done on $t$-bar cores has assumed $t$ is odd, because the combinatorics is not as nice when $t$ is even. For example if $t$ is even a partition may be a $t$-bar core without being a $2t$-bar core (e.g. for $\lambda$= $(3,1)$ and $t=2$ one can remove a $4$-bar but not a $2$-bar). Also the number of partitions with a given $t$-bar core and $t$ bar weight may depend on the choice of $t$-bar core and not just the weight. Neither of these pathologies happens for ordinary $t$-cores and rim $t$-hooks. However, as Olsson observes in \cite[p.235]{OlssonFrobeniussymbols}, these difficulties do not arise if there are no beads on runner $t/2$. This will be the case in our situation, where $t$ will equal $2e$ for $e$ even, and the partitions will have only odd parts.

\begin{remark}
  \label{rmk:2ebarquotient}
  Suppose $t=2e$ for $e$ even and let $\lambda$ have distinct odd parts. Then the abacus display for $\lambda$ with $t$ runners has beads only on runners with odd labels and, in particular, no beads on runner 0 or $e=t/2$. Moreover all partitions with distinct odd parts and the same set of residues mod $2e$ will have the same $2e$-bar core as $\lambda$.
\end{remark}

We now consider the $t$-bar quotient, following \cite{MorrisYaseenCombinatorialShiftedYoung}. Assume $t=2e$ for $e$ even, and $\lambda$ has distinct odd parts. Then we see that the $t$-bar quotient is a sequence of $e/2$ partitions, one for each pair of runners $\{(i,t-i) \mid i=1,3,5, \ldots, e-1\}$. The actual construction is in \cite[Theorem 2.2]{MorrisYaseenCombinatorialShiftedYoung} but we will need only the count:

\begin{lemma}\cite[Theorem 2.(2)]{MorrisYaseenCombinatorialShiftedYoung}
  \label{lem:Countpartitionswithtbarcores}
  The number of partitions with distinct odd parts having the same $2e$-bar core as $\lambda$ and $2e$-bar weight $\tilde{w}$ is the number of $e/2$ tuples of partitions of total weight $\tilde{w}$.
\end{lemma}

\noindent {\bf Proof of Theorem \ref{thm: kappa}}:
We are now ready to prove Theorem \ref{thm: kappa} by constructing a bijection between the Mullineux fixed points on the right and a set which has cardinality $\kappa(\frac{e}{2},w)$ on the left.

Suppose $\lambda \vdash n$ has $e$-weight $w$ and self-conjugate $e$-core $\lambda_{(e)}.$ Apply Bessenrodt's bijection to get a pair $\{\tau, e\gamma\}$ where $\tau=(c_1,c_2, \ldots, c_k)$ has distinct odd parts and $\gamma$ has all odd parts, and the nonzero residues mod $2e$ for $\tau$ are the same as for $\lambda$. In particular all the $\tau$ have the same $2e$-bar core. As in \cite{BessendrodtOlssonOnResidueSymbols}, we need to  check that that $2e$-bar core  $\tilde{\tau}_{(2e)}$ is the same size as the $e$-core of $\lambda$, i.e. that $|\lambda_{(e)}|=|\tilde{\tau}_{2e}|.$
Equation \ref{eq: weightofpcorefornvector} gives us $|\lambda_{(e)}|$ in terms of the $n_i$. Now the argument on the top of page 238 of \cite{BessendrodtOlssonOnResidueSymbols} proves that the size $|\tilde{\tau}_{(2e)}|$ as a function of the $n_i$ agrees with the formula in \eqref{eq: weightofpcorefornvector}.

Recall that when $e$ is odd  and $\lambda$ has even weight $w$,  all the corresponding partitions had the same $2e$-bar weight $w/2$. In this situation, with $e$ even, the $2e$-bar weight of $\tau$ can be less, with the ``difference" made up for by $e\gamma.$

The next step in the bijection is to calculate the $2e$-bar quotient of $\tau.$ Since $\tau$ has distinct odd parts, its $2e$-bar quotient is a tuple $(\rho^1,\rho^2, \cdots, \rho^{e/2})$ with $2e$-bar weight $\tilde{w}=\sum_i |\rho^i|$ and $$\tau \vdash  |\tilde{\tau}_{2e}|+2e\tilde{w}.$$

Now we can describe the bijection that proves Theorem \ref{thm: kappa}. Given $\lambda$ as above we send $\lambda$ to the tuple $(\rho^1,\rho^2, \cdots, \rho^{e/2}; e\gamma)$. The calculations above show that
\begin{eqnarray*}
  w&=&\frac{n-|\lambda_{(e)}|}{e}\\
   &=& \frac{|\tau|+e|\gamma|-|\tilde{\tau}_{2e}|}{e} \\
   &=& \frac{|\tau|-|\tilde{\tau}_{2e}|}{e}+|\gamma|\\
   &=&2\sum_i |\rho^i|+|\gamma|.
\end{eqnarray*}
So $(\rho^1,\rho^2, \cdots, \rho^{e/2}; e\gamma)$ is in the set enumerated by $\kappa(e/2, w).$ 

The final necessary observation is that any choice of $2e$-bar quotient concentrated on the odd runners together with $e\gamma$ satisfying the weight condition $w=2\sum_i |\rho^i|+|\gamma|$ will correspond to an original $\lambda$ in the correct block. This is again the observation from \cite{BessendrodtOlssonOnResidueSymbols} that adding or removing a $2e$ bar does not change the vector $\hat{n}$, nor does the choice of $e\gamma$. So the map is surjective. Injectivity follow from the fact that Bessenrodt's map is bijection, and that a partition is uniquely determined by its $t$-bar core and $t$-bar quotient for any $t$.

\bibliography{references0824}{}

\def\cprime{$'$}
\begin{thebibliography}{Hum86}

\bibitem[AO91]{AndrewsOlssonPartitionIdentities}
George~E. Andrews and J\o rn~B. Olsson.
\newblock Partition identities with an application to group representation
  theory.
\newblock {\em J. Reine Angew. Math.}, 413:198--212, 1991.

\bibitem[Bes95]{BessenrodtGeneralisationsofAndrewsOlsson}
Christine Bessenrodt.
\newblock Generalisations of the {A}ndrews-{O}lsson partition identity and
  applications.
\newblock {\em Discrete Math.}, 141(1-3):11--22, 1995.

\bibitem[BO98]{BessendrodtOlssonOnResidueSymbols}
C.~Bessenrodt and J.~B. Olsson.
\newblock On residue symbols and the {M}ullineux conjecture.
\newblock {\em J. Algebraic Combin.}, 7(3):227--251, 1998.

\bibitem[Fay22]{FayersCrystalsRegularisation}
Matthew Fayers.
\newblock Crystals, regularisation and the {M}ullineux map.
\newblock {\em J. Comb. Algebra}, 6(3-4):315--352, 2022.

\bibitem[FK97]{FordKleshMullineux}
Ben Ford and Alexander~S. Kleshchev.
\newblock A proof of the {M}ullineux conjecture.
\newblock {\em Math. Z.}, 226(2):267--308, 1997.

\bibitem[GKS90]{GarvanKimStantonCranksandtcores}
Frank Garvan, Dongsu Kim, and Dennis Stanton.
\newblock Cranks and {$t$}-cores.
\newblock {\em Invent. Math.}, 101(1):1--17, 1990.

\bibitem[Hum86]{HumphreysBlocksofProjective}
John~F. Humphreys.
\newblock Blocks of projective representations of the symmetric groups.
\newblock {\em J. London Math. Soc. (2)}, 33(3):441--452, 1986.

\bibitem[Jac23]{JaconMullineuxandcrystal}
Nicolas Jacon.
\newblock Mullineux involution and crystal isomorphisms.
\newblock {\em Comb. Theory}, 3(2):Paper No. 8, 33, 2023.

\bibitem[Jam78]{Jamesbook}
Gordon James.
\newblock {\em The representation theory of the symmetric groups}.
\newblock Number 682 in Lecture Notes in Mathematics. Springer-Verlag, 1978.

\bibitem[JK81]{JamesKerberbook}
Gordon James and Adalbert Kerber.
\newblock {\em The representation theory of the symmetric group}, volume~16 of
  {\em Encyclopedia of Mathematics and its Applications}.
\newblock Addison-Wesley Publishing Co., Reading, Mass., 1981.

\bibitem[Kle96]{KleshchevBranchingIII}
A.~S. Kleshchev.
\newblock Branching rules for modular representations of symmetric groups.
  {III}. {S}ome corollaries and a problem of {M}ullineux.
\newblock {\em J. London Math. Soc. (2)}, 54(1):25--38, 1996.

\bibitem[Mat99]{Mathasbook}
Andrew Mathas.
\newblock {\em Iwahori-{H}ecke algebras and {S}chur algebras of the symmetric
  group}, volume~15 of {\em University Lecture Series}.
\newblock American Mathematical Society, Providence, RI, 1999.

\bibitem[Mor65]{MorrisTheSpinRepresentation}
A.~O. Morris.
\newblock The spin representation of the symmetric group.
\newblock {\em Canadian J. Math.}, 17:543--549, 1965.

\bibitem[Mul79]{MullineuxBijectionsofpregular}
G.~Mullineux.
\newblock Bijections of {$p$}-regular partitions and {$p$}-modular irreducibles
  of the symmetric groups.
\newblock {\em J. London Math. Soc. (2)}, 20(1):60--66, 1979.

\bibitem[MY86]{MorrisYaseenCombinatorialShiftedYoung}
A.~O. Morris and A.~K. Yaseen.
\newblock Some combinatorial results involving shifted {Y}oung diagrams.
\newblock {\em Math. Proc. Cambridge Philos. Soc.}, 99(1):23--31, 1986.

\bibitem[{OEI}]{oeis}
{OEIS Foundation Inc.}
\newblock The {O}n-{L}ine {E}ncyclopedia of {I}nteger {S}equences.
\newblock Published electronically at \url{http://oeis.org}.

\bibitem[Ols87]{OlssonFrobeniussymbols}
J\o rn~B. Olsson.
\newblock Frobenius symbols for partitions and degrees of spin characters.
\newblock {\em Math. Scand.}, 61(2):223--247, 1987.

\bibitem[Ols92]{OlssonTheNumberofModularCharacters}
J\o rn~B. Olsson.
\newblock The number of modular characters in certain {$p$}-blocks.
\newblock {\em Proc. London Math. Soc. (3)}, 65(2):245--264, 1992.

\end{thebibliography}
\bibliographystyle{alpha}

\end{document}